\documentclass[12pt,twoside,singlespace]{amsart}
\usepackage{array, paralist, enumerate, amsmath, amsfonts, amssymb, amscd, color, mathrsfs,comment}
\usepackage{amsthm} 

\usepackage{geometry}
\usepackage{framed}
\usepackage{hyperref}
\usepackage{graphicx}
\usepackage{epstopdf}
\usepackage{tcolorbox}

\usepackage[all,cmtip]{xy}
\usepackage{tikz}
\usepackage{tkz-graph}
\usetikzlibrary{arrows,%
                shapes,positioning}

\definecolor{DarkBlue}{rgb}{0,0,0.8} 
\definecolor{DarkGreen}{rgb}{0,0.5,0.0} 
\definecolor{DarkRed}{rgb}{0.9,0.0,0.0} 

\usepackage[T1]{fontenc}
\usepackage[latin1]{inputenc}

\newtheorem*{thm*}{Theorem}

\numberwithin{equation}{section}
\newtheorem{thm}[equation]{Theorem}
\newtheorem{lem}[equation]{Lemma}
\newtheorem{cor}[equation]{Corollary}
\newtheorem{prop}[equation]{Proposition}

\theoremstyle{definition}

\newcommand{\on}{\operatorname}

\newcommand{\h}{\hat}
\newcommand{\fd}{g}
\newcommand{\fr}{f}

\newcommand{\logical}{\mathbb{U}}
\newcommand{\flip}{\on{flip}}
\newcommand{\negation}[1]{\neg #1}
\newcommand{\agenda}{\overline{X}}
\newcommand{\halfagenda}{X}
\newcommand{\closure}{\on{Closure}}
\newcommand{\suchthat}{\;\ifnum\currentgrouptype=16 \middle\fi|\;}

\newcommand{\arc}[1]{\ar @{-} #1 |-(0.57){\SelectTips{cm}{}\object@{>}}}
\newcommand{\arccurv}[2]{\ar @{-} #2 #1 |-(0.57){\SelectTips{cm}{}\object@{>}}}
\newcommand{\arccurvshift}[3]{\ar @{-} #2 #1 |-(#3){\SelectTips{cm}{}\object@{>}}}

\title{Applying Fourier Analysis to Judgment Aggregation}
\author{Yan X Zhang}
\address{Yan X Zhang, San Jose State University, San Jose, CA, USA}
\email{yan.x.zhang@math.sjsu.edu}

\begin{document}

\pagestyle{plain}

\begin{abstract}
The classical Arrow's Theorem answers ``how can $n$ voters obtain a collective preference on a set of outcomes, if they have to obey certain constraints?'' We give an analogue in the judgment aggregation framework of List and Pettit, answering ``how can $n$ judges obtain a collective judgment on a set of logical propositions, if they have to obey certain constraints?'' We abstract this notion with the concept of \emph{normal pairs} of functions on the Hamming cube, which we analyze with Fourier analysis and elementary combinatorics. We obtain judgment aggregation results in the special case of \emph{symbol-complete} agendas and compare them with existing theorems in the literature. Amusingly, the non-dictatorial classes of functions that arise are precisely the classical logical functions OR, AND, and XOR. 
\end{abstract}

\maketitle

\section{Introduction}
\label{sec:intro}

In \cite{list2002aggregating}, List and Pettit explore the ``doctrinal paradox'' put forth in Kornhauser and Sager \cite{kornhauser1993one} and prove an impossibility theorem. In their followup work \cite{list2004aggregating}, they contrast their theorem with the celebrated Arrow's Theorem from the social choice literature. List-Pettit's Theorem is about the difficulty of aggregating $n$ judges' collective judgments on an agenda of logical propositions into a single judgment, and Arrow's Theorem is about the difficulty of aggregating $n$ voters' preference orderings into a single preference ordering. The ``judgment aggregation'' framework that comes out of List and Pettit's work is very powerful, because logical propositions are very expressive and flexible. 

\cite{list2004aggregating} concludes that while List-Pettit's theorem is superficially similar to Arrow's Theorem, the two theorems are quite different at heart. What, then, \textbf{is} the right analogue to Arrow's Theorem in the judgment aggregation framework? Arrow's Theorem is typically phrased in the following way: ``if $n$ voters wish to obtain a collective preference on a set of outcomes, then it is impossible for the aggregation function to be simultaneously Pareto, Indifferent to Irrelevant Alternatives, and Non-Dictatorial.'' In the judgment aggregation literature, Dietrich and List continue this trend in \cite{list2007} and \cite{list2013}, proving impossibility theorems that cover a large space of situations, including explicitly making an analogy with Arrow's Theorem in \cite{list2007}. 

As \emph{one persons' modus ponens is another's modus tollens}, we explore a different phrasing of Arrow's Theorem, which is: ``if $n$ voters wish to obtain a collective preference on a set of outcomes, and the aggregation function must be Pareto and Indifferent to Irrelevant Alternatives, then it must be a dictatorship function.'' That is, we think of Arrow's Theorem as a \textbf{classification} theorem rather than an impossibility theorem.

In this direction, we define an aggregation function of $n$ judgments of $\agenda$ to be \emph{Arrovian} if it satisfies Arrow's Theorem-like constraints. We classify possible Arrovian aggregation functions in the case when our agenda $\agenda$ of logical propositions is \emph{symbol-complete} and \emph{symbol-connected}, a condition special enough to give us a lot of structure but general enough to include many realistic applications.  We discover a friendly set of functions: the OR, AND, and XOR functions from the first chapter of any undergraduate book on logic or computer science. Here is a simplified paraphrasing of our main theorem; the details and definitions are made explicit in the paper:

\begin{thm*}[Corollary~\ref{cor:main}, paraphrased]
Suppose the agenda $\agenda$ is symbol-complete and symbol-connected, containing the compound proposition $\fd$. If $f$ aggregates the judgments of $n$ judges on $\agenda$ in an Arrovian manner, then either:
\begin{itemize}
\item $f$ is a ``dictatorship'' (that is, $f$ returns the judgment of a single judge), or
\item $\fd$ must be the OR, AND, or XOR function (or negation), and $f$ is an ``oligarchy'' on a distinguished subset $S$ of judges, obtained by performing the \textbf{same} function type (or possibly its negation) as $\fd$ on $S$.
\end{itemize}
\end{thm*} We now give two examples. 
\begin{itemize}
\item if $n$ judges are judging on the agenda $\{$The defendant lied, the defendant misspoke, the defendant committed the crime, the defendant was guilty and gave false information (i.e. committed the crime and then misspoke or lied)$\}$, logically representable as $\{P, Q, R, (P \vee Q) \wedge R\}$, then the only Arrovian method to aggregate the $n$ judgments is for the judges to select a single ``dictator'' judge to represent the entire group.
\item if $n$ judges are judging on the agenda $\{$The defendant committed unspeakable crime 1, the defendant committed unspeakable crime 2, the defendant committed unspeakable crime 3, the defendant committed at least one unspeakable crime$\}$, logically representable as $\{P, Q, R, P \vee Q \vee R\}$, then the Arrovian methods of aggregating the $n$ judgments are for a subset $S$ of the judges to decide that each proposition is judged true in the aggregate if and only if at least one of $S$ think the proposition is true (in other words, an OR function over the judges). Furthermore, if we add to the agenda any other proposition (besides OR's) that uses the same symbols (such as $Q \wedge R$), then only the dictatorship option ($|S| = 1$) is possible. If we add logical OR statements to the agenda (such as $P \vee R$), then the OR function over the judges in $S$ can still be used.
\end{itemize}

While our main results offer original consequences and insight (the most interesting, in our opinion, being that the aggregation function is forced to follow the same logical structure as a compound proposition in $\agenda$, applied to possibly a different number of arguments!), they overlap significantly with existing theorems in \cite{list2013}, so we think of our results primarily as a refinement on the existing work when we have at least one symbol-closed compound proposition. We would not be surprised, for example, if our main results would follow from the existing theorems after some work and case analysis. We believe our paper's main contribution to the literature is actually the act of introducing Fourier analysis techniques, inspired by their success in the closely related social choice framework by works such as \cite{kalai2002}. Our work demonstrates that we can prove strong results fairly efficiently with these methods.

In Section~\ref{sec:prelim}, we give the preliminaries of Arrow's Theorem and List-Pettit's judgment aggregation framework. In Section~\ref{sec:setup}, we abstract the problem and reduce studying aggregation functions to what we call \emph{normal pairs}, which are pairs of functions on the Hamming cube that satisfy a commutation-like property. We believe normal pairs lie at the heart of many social choice / judgment aggregation problems, and may be very useful in future work. In Section~\ref{sec:main}, we state and prove our main result about normal pairs, Theorem~\ref{thm:main}, with Fourier analysis. In Section~\ref{sec:list}, we apply Theorem~\ref{thm:main} to judgment aggregation and put our work in context with the existing literature, comparing it with existing theorems. Finally, in Section~\ref{sec:conclusion}, we conclude with philosophical discussions and ideas for future work.

\section{Preliminaries and Definitions}
\label{sec:prelim}

\subsection{Arrow's Theorem}
In this paper, we call the framework behind Arrow's Theorem the ``social choice framework.'' In this setup, there is a set of $n$ \emph{voters} $[n] = \{1, 2, \ldots, n\}$ voting on a set $X$ of \emph{alternatives}. We present each voter's \emph{preference} $p$ as a \emph{partial order} on $X$, defined as an anti-symmetric, transitive, and reflexive relation. We denote by $P_\geq$ the set of partial orders on $X$. We then define \emph{social preference functions (SPF's)} to be functions $f\colon P_{\geq}^n \rightarrow P_{\geq}$. So a SPF aggregates a profile of $n$ preferences into a single collective preference. 

We now define some properties relevant to a SPF $f$. Suppose $f(p_1, \ldots, p_n) = p$, we define $f$ to be:
\begin{itemize}
\item A SPF $f$ is \emph{Pareto} If for all $i \in N$, $p_i(x) \geq p_i(y)$, then $p(x) \geq p(y)$.
\item A SPF $f$ has \emph{Indifference of Irrelevant Alternatives (IIA)} if $(p_1, \ldots, p_n)$ and $(p_1^*, \ldots, p_n^*)$ are two profiles in the domain of $f$ and there exists $x$ and $y$ such that for all $i$, $p_i(x) \geq p_i(y)$ if and only if $p_i^*(x) \geq p_i^*(y)$, then if $p = f(p_1, \ldots, p_n)$ and $p^* = f(p_1^*, \ldots, p_n^*)$, we must have $p(x) \geq p(y)$ if and only if $p^*(x) \geq p^*(y)$.
\end{itemize}
An enlightening restatement of IIA for our purposes is the following: if we think of a preference $p$ as a function from pairs of alternatives $(x,y)$ to $\{T,F\}$, where we say $p(x,y) = T$ if and only if $p \geq y$, then if $p' = f(p_1, \ldots, p_n)$, it means $p'(x,y)$ only depends on $p_1(x,y), \ldots, p_n(x,y)$. In other words, for all pairs $(x,y)$ there exists a function $p_{x,y}\colon \{T, F\}^n \rightarrow \{T, F\}$ such that $p'(x,y) = p_{x,y}(p_1(x,y), p_2(x,y), \ldots, p_n(x,y))$. We will use this intuition again in the future.

Suppose $f$ is Pareto and IIA. Then we call $f$ \emph{Arrovian}. We rephrase\footnote{Usually, Arrow's Theorem is stated as an ``impossibility theorem.'' That is, it is impossible for a social preference function to be Pareto, IIA, and \textbf{not} a dictatorship. Because our main result is more of a ``classification theorem,'' that is, listing the types of functions (including dictatorships) that satisfy Pareto and IIA, we do not use the ``impossibility'' presentation.} the classical Arrow's theorem in the following form:
\begin{thm}[Arrow's Theorem, \cite{arrow1950}, paraphrased]
\label{thm:arrow}
Let $|X| > 2$. Suppose $f$ is an Arrovian SPF, then $f$ is a dictatorship (that is, $f(p_1, \ldots, p_n) = p_i$ for some $i$).
\end{thm}

\subsection{The Judgment Aggregation Framework}

In the judgment aggregation framework of List and Pettit \cite{list2002aggregating}, $n$ \emph{judges} $[n] = \{1,2,\ldots,n\}$ seek to judge an agenda (set of logical propositions). In this paper, we define an \emph{agenda} $\agenda$ to be a subset of the boolean algebra\footnote{In the treatment of \cite{list2007} and others, the agenda is defined by an arbitrary set related by rules about which subsets of propositions would imply others, without the need of an underlying set of logical symbols. Any such set can actually be embedded in a boolean algebra and we choose to do so in this paper.} of a \emph{symbol set} of logical symbols such that $\agenda$ is closed under negation (i.e. for all $x \in \agenda$, we must have $\negation{x} \in \agenda$). Given $x \in \agenda$, we call $x$ \emph{atomic} if $x$ or $\negation{x}$ belongs to the symbol set and $x$ \emph{compound} otherwise.

We usually describe the agenda by a \emph{basis} $\halfagenda$, which contains exactly one of each of $x$ and $\negation{x}$ for every $x \in \agenda$; note we can recover the agenda $\agenda$ from the basis $\halfagenda$ by adding negations, much in the same way that we can recover a vector space from the basis by allowing multiplication by scalars. For example, a possible symbol set could be $\{P,Q\}$, with $\halfagenda = \{P, \negation{(P \wedge Q)}, (P \vee Q)\}$ and $\agenda = \{P, \negation{P}, (P \wedge Q), \negation{(P \wedge Q)}, (P \vee Q), \negation{(P \vee Q)}\}$. Each judge's \emph{judgment} can then be defined as an assignment of $T$ (True) or $F$ (False) to each proposition in the agenda; formally, a function $\agenda \rightarrow \{T,F\}$. 

If the judges are to make judgments on a proposition, it seems natural for them (at least some of the time) to also be making judgments on its symbols. For example, if the judges were to judge on ``the defendant committed crime A or crime B,'' we would assume they would also argue about whether the defendant committed the individual crimes separately. We say that a compound proposition is \emph{symbol-closed} if all of its symbols are in $\agenda$. For example, in $\halfagenda = \{P, \negation{Q}, \negation{P} \wedge Q\}$, $(\negation{P} \wedge Q)$ is symbol-closed. However, in $\{Q, P \wedge Q\}$, $(P \wedge Q)$ is not symbol-closed.

We say that a set of propositions $\halfagenda$ is \emph{symbol-complete} if the symbol set is in $\agenda$.  Note that when $\fd \in \agenda$ is symbol-closed, $\fd$ and its symbols form a symbol-complete subset of $\agenda$. We denote this subset by $\closure(\fd)$. For example, $\closure(P \wedge Q) = \{P, Q, P \wedge Q\}$.

In this paper, we typically assume our agenda is symbol-complete. This is a very strong assumption. The existing results in \cite{list2007} and \cite{list2013} do not make this assumption and still hold for very general agendas. Compared to these nice results, our restriction gives results that are less generally applicable but stronger when applied. We discuss this trade-off more in Sections~\ref{sec:list} and \ref{sec:conclusion}.

In the social choice framework, the rationality (i.e. if $x \geq y$, $y \geq z$, then $x \geq z$) of a preference is embedded in the requirement that the preference is a partial order. In the judgment aggregation framework, we analogously define $\logical_{\halfagenda}$, the \emph{fully rational} judgments on $\agenda$, to be the set of judgments on $\halfagenda$ such that there exists some assignment of truth values to the symbol set such that each proposition in the judgment (as a function of the symbol set) is consistent with the assignment; for example, a fully rational judgment on $(Q, P \vee Q, P \wedge R)$ could be $\{Q:T, (P \vee Q):T, (P \wedge R):F\}$, corresponding to $(Q=T, P=F, R=F)$. Note that in a fully rational judgment, different truth values must be assigned to $x$ and $\negation{x}$ for all $x \in \halfagenda$. From this, we can see that while we care about the agenda $\agenda$, in $\logical_{\halfagenda}$ it suffices to provide judgment on just the basis $\halfagenda$; in other words, $\logical_{\halfagenda}$ is well-defined. Thus, we usually look at just the judgment restricted to $\halfagenda$ instead of $\agenda$ to avoid redundancy, losing no information.  We call the set of functions $f\colon \logical_{\halfagenda}^n \rightarrow \logical_{\halfagenda}$ \emph{judgment aggregation rules (JAR's) on $\halfagenda$} to aggregate fully rational judgments, as an analogue of social preference functions being used to aggregate non-cyclic preferences. In \cite{list2007}, JAR's are simply called ``aggregation rules.'' We add ``judgment'' to disambiguate the many uses of ``aggregation'' in this paper.

In a symbol-complete $\halfagenda$ with $m$ atomic propositions, there is a bijection between $\logical_{\halfagenda}$ and the $2^m$ assignments of $\{T, F\}$ to the atomic propositions, as every assignment of the atomic propositions uniquely extends to a fully rational assignment to the propositions in $\halfagenda$. 

\subsection{Logic}

We define two propositions (or sets) in $\halfagenda$ to be \emph{symbol-disjoint} if the logical symbols the two propositions (or sets) use are disjoint. As an example, $A \wedge B$ is symbol-disjoint from $C$, but not from $B \wedge C$. Consider a graph $G_{\halfagenda}$ where the vertices are the propositions and we draw an edge between two propositions if and only if they share a symbol. We say that $\halfagenda$ is \emph{symbol-connected} if $G_{\halfagenda}$ is a connected graph. As an example, if $\halfagenda = \{P,Q,R,S, P\wedge Q, R \vee S\}$, the set $\{P, Q, P \wedge Q\}$ is symbol-disjoint from the set of the other $3$ propositions, so the graph $G_{\halfagenda}$ splits into two connected components and is therefore not symbol-connected.

Given a set $S$, let $\{T, F\}^S$ be the set of $\{T, F\}$-valued functions on $S$. One of our main ideas (which lets us use Fourier analysis) is that we interchangeably think of a proposition as an element of the boolean algebra and as a $\{T, F\}$-valued function on $\{T, F\}^S$, where $S$ are the symbols in the proposition. We now make some relevant notations and definitions:
\begin{itemize}
\item For any function $g\colon D \rightarrow \{T, F\}$ where $D \subset \{T, F\}^S$, we say that \emph{$g$ is ($s$-) pivotal at $v$} if $s \in S$, $v \in \{T, F\}^S$, $v'$ is defined to be $v$ with the $s$-indexed element flipped, and $g(v) \neq g(v')$. Note this requires both $v$ and $v'$ to be in the domain $D$. Also note that $g$ is $s$-pivotal at $v$ if and only if it is $s$-pivotal at $v'$ as well. 
\item In the setup above, we say that $s$ is \emph{irrelevant} for the function $g$ if $g$ is not $s$-pivotal anywhere. We say $s$ is \emph{relevant} otherwise.
\end{itemize}

Note that we can assume that any symbol appearing in a compound proposition $\fd \in \agenda$ must be relevant for $\fd$ (for example, $P$ is irrelevant to $\fd = (P \wedge \negation{P}) \vee Q \vee R$, but we can simplify $\fd$ to $Q \vee R$ in this case). This removes some pathological cases.

We say that $\{v_s\}_{s \in S}$ form a set of \emph{consistent} values for a set of propositions $S \subset \halfagenda$ if there exist assignments of $\{T, F\}$ to the symbol set such the induced logical values of each $s \in S$ equals $v_s$.  We denote the set of sets of consistent values for $S$ by $Cons(S)$. Equivalently, a set is in $Cons(S)$ if it is the restriction of some element in $\logical_{\halfagenda}$ to $S$. For a proposition $c$ and a set of propositions $S \subset \halfagenda$, we say that $c$ is \emph{determined by $S$} if $c$ can be written as a function from $Cons(S)$ to $\{T, F\}$. For example, $P$ and $\negation{(Q \wedge R)}$ determine $P \vee (Q \wedge R)$. A recurring theme is that a compound proposition is determined by its symbols.

\subsection{Fourier Analysis on the Hamming Cube}

We present some elementary Fourier analysis tools in style of \cite{kalai2002} to use in Section~\ref{sec:main}. See e.g. \cite{odonnell2014} for a survey, including what we use here and beyond.

Given a function $f\colon \{1, -1\}^n \rightarrow \{1, -1\}$, we can decompose the function as 
\[
f(s_1, \ldots, s_n) = \sum_{R \subset [n]} \h{f}(R) s_R,
\]
where we define $s_R = \prod_{r \in R} s_r$. These \emph{Fourier coefficients} $\h{f}(R)$ satisfy $\sum_{R \subset [n]} \h{f}^2(R) = 1$. 

Given the desired values of $f$, we can solve for $f$ via interpolation because the function $\prod \frac{1}{2^n}(1 \pm x_1)(1 \pm x_2)\cdots (1 \pm x_n)$ returns $1$ on a single list $(x_1, \ldots, x_n) \in \{1, -1\}^n$ and $0$ for all other lists. Doing so for a few key functions gives:
\begin{itemize}
\item the constant function is the function $\pm 1$, so its only nonzero Fourier coefficient is $\h{f}(\emptyset) = \pm 1$. 
\item the XOR function is the function $f(s_1, \ldots, s_n) = (-1)^ns_1 s_2 \cdots s_n$, so its only nonzero Fourier coefficient is $\h{f}([n]) = (-1)^n$. 
\item the NXOR function is the negation of the XOR function, or $f(s_1, \ldots, s_n) = (-1)^{n+1}s_1 s_2 \cdots s_n$. It also only has a single nonzero Fourier coefficient $\h{f}([n]) = (-1)^{n+1}$.
\item the AND function is the function $-1 + \frac{1}{2^{n-1}}(1+s_1)\cdots (1+s_n)$, so for all $R \neq \emptyset$, the Fourier coefficient is $\h{f}(R) = \frac{1}{2^{n-1}}$. For the empty set we must subtract the $1$ in the beginning to get $\h{f}(\emptyset) = \frac{1}{2^{n-1}} - 1$. 
\item the OR function is the function $1 - \frac{1}{2^{n-1}}(1-s_1)\cdots (1-s_n)$, so for all $R \neq \emptyset$, the Fourier coefficient is $\h{f}(R) = (-1)^{|R|+1}\frac{1}{2^{n-1}}$. For the empty set we must add the $1$ in the beginning to get $\h{f}(\emptyset) = 1 - \frac{1}{2^{n-1}}$. 
\end{itemize}

Given a function $f\colon \{-1, 1\}^n \rightarrow \{-1, 1\}$, for $i \in \{1, 2, \ldots, n\}$ and $x, y \in \{1, -1\}$, we say that $f$ \emph{forces $y$ with $x$ at index $i$} if whenever $r = (r_1, \ldots, r_{n})$ with $r_i = x$, we have $f = y$. In this case, we say argument $i$ is \emph{forceable} for $f$. Otherwise, we say index $i$ is \emph{free} for $f$. We say such a function $f$ is \emph{forceful} if $f$ is forceable at every index. 

\begin{lem}
\label{lem:forceful}
Suppose $f\colon \{-1, 1\}^n \rightarrow \{-1, 1\}$ where $(n > 1)$ is a non-constant forceful function. Then it can be written as 
\[
f(r_1, \ldots, r_n) = -c_0 + c_0 \frac{1}{2^{n-1}} (1 + c_1r_1) (1 + c_2r_2) \cdots (1 + c_nr_n)
\]
for arbitrary choices of $c_0, c_1, \ldots, c_n \in \{-1,1\}$. In this case, $\h{f}(S) = c_0(\prod_{s \in S}c_s)/2^{n-1}$ for all $S \neq \emptyset$, and $\h{f}(\emptyset) = c_0(1/2^{n-1}-1)$.
\end{lem}
\begin{proof}
If $f$ is forceful, then $f$ forces some $y_i \in \{-1, 1\}$ with $x_i \in \{-1,1\}$ at index $i$ for all $i$. We must have all the $y_i$'s equal to a constant $y$, otherwise we can force two contradictory results since $n \geq 2$ (if $x_i$ forces $y_i$ and $x_j$ forces $\negation{y_i}$, then we get a contradiction when looking at any set of arguments including $x_i$ and $x_j$ in the $i$-th and $j$-th spots, respectively). We now know that $f$ takes the same value $y$ at every argument but one (namely, when we pick $-x_i$ for each index $i$. Linear interpolation of $f$ at the $x_i$ then gives the above form.
\end{proof}

In particular, note that OR and AND are forceful functions. For example, the OR function forces $1$ (True) if any argument equals $1$, and can be written $1-\frac{1}{2^{n-1}} (1-r_1)\cdots(1-r_n)$. 

\subsection{The Main Question}

Since our goal is an analogue of Arrow's Theorem, we now give some analogous definitions in the judgment aggregation framework to the concepts used in Arrow's Theorem. We follow the nomenclature of \cite{list2013}.

  
\begin{itemize}
\item A JAR $f$ is \emph{Unanimity-Preserving (UP)} if for every proposition $x \in \halfagenda$, if for all $i \in [n]$, $p_i(x) = T$ (respectively $F$), then $f(p_1, p_2, \ldots, p_n)(x) = T$ (respectively $F$). 
\item A JAR $f$ is \emph{Propositionally Independent (PI)} if for each proposition $x \in \halfagenda$ there is a function $f_x$ such that 
\[
f(p_1, \ldots, p_n)(x) = f_x(p_1(x), p_2(x), \ldots, p_n(x))
\]
for all $(p_1, \ldots, p_n)$. 
\end{itemize}
These are completely analogous to Pareto and IIA, respectively. As in Section~\ref{sec:prelim}, we say that a judgment aggregation rule is \emph{Arrovian} if it is both UP and PI. We mention a couple of similar notions from \cite{list2002aggregating}:
\begin{itemize}
\item A JAR $f$ has \emph{Anonymity} if $f$ is invariant under permutations on the judges. In other words, rearranging the arguments $p_i$ of $f(p_1, \ldots, p_n)$ does not change the judgment.
\item A JAR $f$ has \emph{Systematicity} if there exists a function $f'\colon \{0,1\}^n \rightarrow \{0,1\}$ such that, for any proposition $x \in \agenda$, $f(p_1, \ldots, p_n)(x) = f'(p_1(x), \ldots, p_n(x))$.
\end{itemize}
These were originally used in List-Pettit's impossibility theorem in \cite{list2002aggregating}. In Section~\ref{sec:list}, we look at some consequences of our main result when an Arrovian JAR is also forced to have these properties.

The main question, then, is:

\begin{tcolorbox}
For a symbol-complete basis $\halfagenda$, what Arrovian JAR's of $\halfagenda$ exist?

\end{tcolorbox}

\section{Setting up the Main Problem}
\label{sec:setup}

We give a few definitions and Lemmas that help us reduce the problem into smaller parts in prepration for Section~\ref{sec:main}. If $f$ is an JAR with $n$ judges on a basis $\halfagenda$, then for any subset $\halfagenda' \subset \halfagenda$, restricting each judgment from $\halfagenda$ to $\halfagenda'$ gives a function $f'\colon \logical_{\halfagenda'}^n \rightarrow \logical_{\halfagenda'}$, which we call the \emph{restriction of $f$ to $\halfagenda'$}. We make a few observations:
\begin{itemize}
\item For any $\halfagenda' \subset \halfagenda$, if $f$ is an Arrovian JAR on $\halfagenda$, the restriction of $f$ to $\halfagenda'$ is also Arrovian.
\item If a set of propositions $\halfagenda$ is not symbol-connected, then $G_{\halfagenda}$ decomposes into a disjoint union $G_{\halfagenda_1} \cup \cdots \cup G_{\halfagenda_k}$ of symbol-connected components, and the set of Arrovian JAR's for $\halfagenda$ is in bijection with the cartesian product of the sets of Arrovian JAR's on $\halfagenda_1, \halfagenda_2, \ldots, \halfagenda_k$. 
\item In the above decomposition, if $\halfagenda$ is symbol-complete, then each of the $\halfagenda_i$ is also symbol-complete.
\end{itemize}

The proofs of these observations are obvious and we omit them. The first observation allows us to freely consider subsets of our agenda without worrying about $f$. In particular, we mostly apply this when we analyze $\closure(\fd)$ for a symbol-closed compound proposition $\fd$. The second and third observations say that instead of doing classifying symbol-complete $\halfagenda$ directly, we can (and should) limit our attention to symbol-connected subsets of $\halfagenda$, who themselves must be symbol-complete. 

\begin{lem}
\label{lem:connected-determined}
If $\halfagenda$ is symbol-complete and symbol-connected, then every compound proposition $\fd \in \halfagenda$ is determined by $\halfagenda \backslash \fd$.
\end{lem}
\begin{proof}
As $\fd$ is compound, it is a function of its symbols and thus determined by its symbols. As $\halfagenda$ is symbol-complete, $\halfagenda \backslash \fd$ contains all the symbols of $\fd$ and thus determines $\halfagenda$. 
\end{proof}

\begin{cor}
If $\fd$ is symbol-closed in $\halfagenda$, then $\fd$ is determined by $\closure(\fd) \backslash \fd$.
\end{cor}

We warn the reader that the adjectives ``symbol-complete'' and ``compound'' in Lemma~\ref{lem:connected-determined} are necessary, suggesting that the problem is much harder without them. If $\halfagenda$ were not symbol-complete, $\halfagenda \backslash \fd$ does not have to determine $\fd$; e.g. $\halfagenda = \{A \vee B, B \vee C\}$. Even in the symbol-complete $\halfagenda = \{A, B, A \vee B\}$, the atomic proposition $A$ is not determined by the other propositions (though e.g. if $B = F$ and $A \vee B = F$, we can deduce $A = T$, but if $B = T$ and $A \vee B = T$ we know nothing about $A$).

We now give an important and rather surprising lemma that does not depend on symbol-completeness. Given a function $f\colon \{T, F\}^n \rightarrow \{T, F\}$, define $\flip(f)$ to be the function $f'$ with the same domain and codomain such that $f'(s_1, \ldots, s_n) = \negation{f(\negation{s_1}, \ldots, \negation{s_n})}$ for all $\{s_i\}$. Note that $\flip(\flip(f)) = f$, so $\flip$ is an involution on the space of such functions.
\begin{lem}
\label{lem:PI-doctrinal-strong}
Let $y \in \halfagenda$ depend on  $x \in \halfagenda$. Furthermore, assume that $y$ is determined by $\halfagenda \backslash y$. If $f$ is an Arrovian JAR on $\halfagenda$, we must have $f_y = f_x$ or $f_y = \flip(f_x)$.
\end{lem}
\begin{proof}
Since $y$ is dependent on $x$, $y$ must be $x$-pivotal somewhere. In other words, there exists consistent judgments $\pi$ and $\pi'$ such that $\pi(x) = T$, $\pi'(x) = F$, and $\pi(y) \neq \pi'(y)$, but for all other propositions $a' \in \halfagenda$, $\pi(a') = \pi'(a')$. 

Now, fix $(s_1, \ldots, s_n) \in \{T, F\}^n$. For each $i$, if $s_i = T$ or $F$ respectively, let $p_i = \pi$ or $\pi'$ respectively. Let $p = f(p_1, \ldots, p_n)$ be the aggregated judgment. For example, the following illustrates the aggregation for $(s_1, \ldots, s_5) = (T, F, T, F, T)$:
\[
\begin{matrix}
\pi(a_1) & \pi(a_1) & \pi(a_1) & \pi(a_1) & \pi(a_1) & \rightarrow & \pi(a_1) \\
\pi(a_2) & \pi(a_2) & \pi(a_2) & \pi(a_2) & \pi(a_2) & \rightarrow & \pi(a_2) \\
\vdots & \vdots & \vdots & \vdots & \vdots & \rightarrow & \vdots \\
T & F & T & T & F & \rightarrow & f_x(T, F, T, T, F) \\
\vdots & \vdots & \vdots & \vdots & \vdots & \rightarrow & \vdots \\
\pi(a_{m}) & \pi(a_{m}) & \pi(a_m) & \pi(a_{m}) & \pi(a_{m}) & \rightarrow & \pi(a_m) \\
\downarrow & \downarrow & \downarrow & \downarrow & \downarrow & & \downarrow \\
w = \pi(y) & \negation{w} = \negation{\pi'(y)} & w & w & \negation{w} & \rightarrow & f_y(w, \negation{w}, w,w, \negation{w})
\end{matrix}
\]

By construction, all propositions not equal to $x$ or $y$ are identically judged by all the judgments $p_i$. By UP, this means $p(z) = \pi(z) = \pi'(z)$ for all propositions $z \notin \{x,y\}$. By the constraint that judgments must be logically consistent, if $p(x) = T$ we must have $p = \pi$ (because $p$ and $\pi$ match on all the non-$y$ propositions and $y$ is determined by the other propositions) and if $p(x) = F$ we must have $p = \pi'$. By construction, the truth values of $x$ and $y$ in $(p_1, \ldots, p_n)$ will either be always equal (in which case their values in the aggregation are also equal) or always opposite (in which case their values in the aggregation are also opposite). As we can do this for all $2^n$ choices of $(s_1, \ldots, s_n)$, we see that for any input in $\{T, F\}^n$, the function $f_y$ is always consistently equal to $f_x$ (when $\pi(y) = T$) or the negation of $f_x$ applied to the negation of its input (when $\pi(y) = F$), so our result is proved.
\end{proof}

\begin{cor}
\label{cor:strong}
Let $\halfagenda$ be symbol-complete and symbol-connected and $f$ be an Arrovian JAR on $\halfagenda$. Then for every two propositions $x, y \in \halfagenda$, we must have $f_y = f_x$ or $f_y = \flip(f_x)$.
\end{cor}
\begin{proof}
Consider every instance of a symbol $x$ (which must be in $\halfagenda$ since it is symbol-complete) that appears in a compound proposition $y$. By Lemma~\ref{lem:PI-doctrinal-strong} and Lemma~\ref{lem:connected-determined}, we have $f_y = f_x$ or $\flip(f_x)$. Because $G_{\halfagenda}$ is symbol-connected, there is a path between every two propositions using only edges between symbols and compounds (no two symbols share a symbol, and if $2$ compounds share a symbol, they must each have an edge to the symbol). Thus, because $\flip()$ is an involution, any two propositions must actually be connected by $0$ or $1$ $\flip()$ operation, which is what we desired.
\end{proof}

\begin{cor}
\label{cor:nosigns}
Let $\halfagenda$ be symbol-complete and symbol-connected and there exists an Arrovian JAR $f$ on $\halfagenda$. Then there exists $\halfagenda'$ representing the same agenda as $\halfagenda$ and $f_x = f_y$ for all $x, y \in \halfagenda'$. 
\end{cor}
\begin{proof}
First, note that replacing a proposition $x$ in $\halfagenda$ by $\negation{x}$ preserves $\agenda$, but $f_{\negation{x}} = \flip(f_x)$. For any $y \in \halfagenda$, define $\halfagenda_1$ to be all the elements $x \in \halfagenda$ (including $y$) with $f_x = f_y$ and define $\halfagenda_2$ to be all the elements $x \in \halfagenda$ with $f_x = \flip(f_y)$. Now, we may replace all the propositions $x \in \halfagenda_2$ by $\negation{x}$. The resulting set of propositions $\halfagenda'$ has the same agenda as $\halfagenda$, but now has $f_x = f_y$ for all $x, y \in \halfagenda'$.
\end{proof}

Corollary~\ref{cor:strong} is very strong; it tells us that we can restrict our attention to $\halfagenda$ where all the $f_x$'s guaranteed by the PI condition on an Arrovian JAR $f$ are actually the same function, which means we almost have Systematicity\footnote{A similar phenomenon happens in proofs of Arrow's Theorem (see \cite{geanakoplos} for a short example), where potentially different aggregation functions on each ``component'' are forced to be identical. In the judgment aggregation framework, equivalent ideas have also been found in e.g. \cite{list2007} and \cite{list2013}.} (the difference between our situation and Systematicity is that on $\agenda\backslash \halfagenda$, the functions are all flips of the functions in $\halfagenda$), despite assuming only PI. If this happens, we say that the $f$ is in \emph{normal form} for $\halfagenda$ and denote, with abuse of notation, the unified function $\fr\colon \{T, F\}^n \rightarrow \{T,F\}$ (that is, $f_x(s_1, \ldots, s_n) = \fr(s_1, \ldots, s_n)$ for all $x \in \halfagenda$ and $(s_1, \ldots, s_n) \in \{T, F\}^n$). 

We now have a very clean and symmetric way to think about how an Arrovian JAR $f$ interacts with a symbol-closed compound proposition $\fd$ in $\halfagenda$. $\closure(\fd) = \{a_1, \ldots, a_m, \fd\}$ is a symbol-closed and symbol-connected subset of $\agenda$. Arrange the $n$ judgments on the $m$ propositions $a_i$ in a $m \times n$ matrix $M$, where $M_{ij} = p_j(a_i)$ is the judgment of judge $j$ on proposition $a_i$. We can think of $\fd$ as a function $\{T, F\}^m \rightarrow \{T, F\}$ that takes the values of the $a_i$ as input, applied to each column. Because $f$ is in normal form, we have a function $\fr\colon \{T, F\}^n \rightarrow \{T, F\}$ applied to each row.  In this setup, we are asking the following two operations to give the same answer:
\begin{itemize}
\item For each row $i$, we compute $\fr(M_{i1}, M_{i2}, \ldots)$ to obtain aggregate propositional judgments $j_1, j_2, \ldots, j_m$, then compute $\fd(j_1, j_2, \ldots, j_m)$.
\item For each column $j$, we compute $\fd(M_{1j}, M_{2j}, \ldots)$ to obtain aggregate compound judgments $k_1, k_2, \ldots, k_n$, then compute $\fr(k_1, k_2, \ldots, k_n)$.
\end{itemize}
We make one more reduction. Recall that we assumed all the symbols in $\fd$ are relevant to $\fd$. Similarly, we can assume all judges are relevant for $f$ by removing irrelevant judges. To obtain solutions for the general case, all we have to do is to add an arbitrary number of irrelevant judges back. In fact, a dictatorship is just the identity function\footnote{This observation allows us to think about Arrow's Theorem in another way, which is ``if we have at least $3$ alternatives, then the only Arrovian SPF $f$ where all the voters are relevant is the identity function $\{T, F\} \rightarrow \{T, F\}$.''} with any number of irrelevant judges! Thus, we say that a pair of boolean functions $\fd\colon \{T, F\}^m \rightarrow \{T,F\}$ and $\fr\colon \{T, F\}^n \rightarrow \{T, F\}$ form a \emph{normal pair} $(\fd, \fr)$ if:
\begin{itemize}
\item For all sets of $M_{ij} \in \{T, F\}$, we have the equation
\[
\fr(\fd(M_{11}, \ldots, M_{m1}), \ldots, \fd(M_{1n}, \ldots, M_{mn})) = \fd(\fr(M_{11}, \ldots, M_{1n}), \ldots, \fr(M_{m1}, \ldots, M_{mn})).
\]
\item Every index $i \in [m]$ is relevant for $\fd$.
\item Every index $j \in [n]$ is relevant for $\fr$.
\item Neither $\fd$ nor $\fr$ is the constant function (this restriction does not rule out anything interesting for our purposes. We want $\fr$ to be UP, so it cannot be constant. We want $\fd$ to be a compound proposition, which also cannot be constant if all its symbols are relevant).
\end{itemize}

Thus, studying compound propositions really come down to studying normal pairs. To be explicit, when we have an Arrovian JAR $f$ on a (not necessarily symbol-complete or symbol-connected!) basis $\halfagenda$, each symbol-closed compound proposition $\fd \in \halfagenda$ gives a $\closure(\fd) \subset \halfagenda$, on which we can assume $f$ is in normal form. We then know $(\fd, \fr)$ is a normal pair where $\fr$ is UP. Our main question is, then, the following slight generalization (removing $\fr$'s UP requirement):

\begin{tcolorbox}
What normal pairs $(\fd, \fr)$ exist?
\end{tcolorbox}

\section{The Main Theorem}
\label{sec:main}

We now classify normal pairs $(\fd, \fr)$ with the aid of some elementary Fourier analysis\footnote{We have also obtained the relevant results, Propositions~\ref{prop:XOR} and \ref{prop:AND-OR}, with purely combinatorial proofs, but we decided to showcase the Fourier analysis proofs because our emphasis is on the techniques being new.} on Boolean functions. In this section, we think of $T$ as $1$ and $F$ as $-1$, so we consider $\fd$ and $\fr$ as functions $\fd\colon \{1, -1\}^m \rightarrow \{1, -1\}$ and $\fr\colon \{1,-1\}^n \rightarrow \{1,-1\}$ respectively. 

As before, we encode the normal pair as a $(m+1) \times (n+1)$ matrix where each cell in the last row corresponds to $\fd$ applied to the cells above it and each cell in the last column corresponds to $\fr$ applied to the cells to the left of it. Being a normal pair, the last row and the last column are consistent with $\fr$ and $\fd$ respectively. 

\[
\begin{matrix}
M_{11} & M_{12} & \cdots & M_{1n} & \rightarrow & x_1 \\
M_{21} & M_{22} & \cdots & M_{2n}  & \rightarrow & x_2 \\
\vdots & \vdots & \ddots & \vdots & \rightarrow & \vdots \\
M_{m1} & M_{m2} & \cdots & M_{mn}  & \rightarrow & x_m \\
\downarrow & \downarrow&  & \downarrow & &  \downarrow \\
y_1 & y_2 & \cdots & y_n & \rightarrow & z.
\end{matrix}
\]

\begin{lem}
\label{lem:powerful-sauce}
Suppose $(\fd, \fr)$ form a normal pair. Then consider any subset $U \subset [m] \times [n]$. For each $r \in [m]$, let $Y_r = \{s \in [n] \suchthat (r,s) \in U\}$, and for each $s \in [n],$ let $X_s = \{r \in [m] \suchthat (r,s) \in U\}$. Now, define
\[
R_U = \{r \in [m] \suchthat \text{for some } y, (r, y) \in U\}, S_U = \{s \in [n] \suchthat \text{for some } x, (x, s) \in U\}.
\]

Then we have 
\[
(\sum_{S' \supset S_U} \h{\fr}(S')\h{\fd}(\emptyset)^{|S'-S_U|}) (\prod_{s \in S_U} \h{\fd}(X_s)) = (\sum_{R' \supset R_U} \h{\fd}(R')\h{\fr}(\emptyset)^{|R'-R_U|})(\prod_{r \in R_U} \h{\fr}(Y_r)).
\]
\end{lem}
\begin{proof}
For a subset $U \subset [m] \times [n]$, we fill the upper-left $[m] \times [n]$ submatrix of $M$ with $1$'s and $(-1)$'s such that we place a $1$ at $M_{ij}$ if and only if  $(i, j) \in U$. For a normal pair, we need:
\[
\fr(\fd(M_{11}, \ldots, M_{m1}), \ldots, \fd(M_{1n}, \ldots, M_{mn})) = \fd(\fr(M_{11}, \ldots, M_{1n}), \ldots, \fr(M_{m1}, \ldots, M_{mn})).
\]
Expanding the left-hand side, we get
\begin{align*}
& \fr(\sum_{R \subset [m]} \h{\fd}(R)\prod_{r \in R}M_{r1}, \ldots, \sum_{R \subset [m]} \h{\fd}(R)\prod_{r \in R}M_{rn}) \\
= & \sum_{S \subset [n]} \h{\fr}(S) \prod_{s \in S} (\sum_{R} \h{\fd}(R) \prod_{r \in R} M_{rs}).
\end{align*}
Consider the coefficient of the term $M_U = \prod_{(i,j) \in U} M_{ij}$ in this expression. In the outermost sum, we must select $S$ such that $S \supset S_U$, otherwise it would be impossible to get all the terms. After a $S \supset S_U$ is selected, for each $s \in S$ in the product, we must pick $R = \emptyset$ for $s \notin S_U$ and $R = X_s$ for $s \in S_U$. Each $\emptyset$ contributes $\h{\fd}(\emptyset)$ and each $X_s$ contributes $\h{\fd}(X_s)$. As the same $X_s$ for each choice of $S$, they can be factored out, and we obtain the left-hand side of our expression. Looking for the coefficient of $M_U$ in the right-hand side of the original equation gives a similar expression. As the $M_U$ form a basis of functions $\{1, -1\}^{[m] \times [n]} \rightarrow \{1, -1\}$, the two coefficients must equal for all $U$, which gives our result.
\end{proof}


Picking an $R \subset [m]$ and $S \subset[n]$ and letting $U$ be the rectangle $R \times S$ in Lemma~\ref{lem:powerful-sauce}, we immediately obtain:
\begin{cor}
\label{cor:powerful}
Suppose $(\fd, \fr)$ form a normal pair. Then for all nonempty $R \subset [m]$ and $S \subset [n]$,
\[
(\sum_{S' \supset S} \h{\fr}(S')\h{\fd}(\emptyset)^{|S'-S|})\h{\fd}(R)^{|S|}  = (\sum_{R' \supset R} \h{\fd}(R')\h{\fr}(\emptyset)^{|R'-R|})\h{\fr}(S)^{|R|} .
\]
\end{cor}

Note that if either $m$ or $n$ equals $1$, the problem is almost trivial: one function (the function with a single argument) must be the identity or the negative identity function (to make its only index relevant, the function cannot be constant), and the other function can be any function where every index is relevant. Thus, it suffices to consider only the case where $m > 1$ and $n > 1$. We call these \emph{nontrivial} normal pairs.

\begin{prop}
\label{prop:XOR}
Suppose $(\fd, \fr)$ form a nontrivial normal pair and $\fd$ is the XOR or NXOR function. Then $\fr$ is also the XOR or NXOR function.
\end{prop}
\begin{proof}
Recall that $\fd$ has $\h{\fd}([m]) = \pm 1$ and all other coefficients equal to $0$. Take a maximal (under set inclusion) $S \subset [n]$ such that $\h{\fr}(S) \neq 0$. This means summing $S' \supset S$ in Corollary~\ref{cor:powerful} only picks up $\h{\fr}(S)$. Corollary~\ref{cor:powerful} applied to $R = [m]$ and $S$ then gives $\h{\fr}(S) = \pm \h{\fr}(S)^{m}$. Since $m > 1$, $\h{\fr}(S) = \pm 1$. Thus, $\h{\fr}(S)$ must be the only nonzero Fourier coefficient in $\h{\fr}$. Because every index is relevant in $\h{\fr}$, we must have $S = [n]$, else there would be a symbol not appearing in the Fourier expansion of $\fr$. This shows that $\fr$ is indeed the XOR function (or its negation).
\end{proof}

We are now ready to justify why we singled out XOR with Proposition~\ref{prop:XOR}:
\begin{prop}
\label{prop:XOR-or-free}
Suppose $(\fd, \fr)$ form a nontrivial normal pair. Then one of these must hold:
\begin{itemize}
\item Both $\fd$ and $\fr$ are the XOR or NXOR functions.
\item Both $\fd$ and $\fr$ are forceful functions.
\end{itemize}
\end{prop}
\begin{proof}
To start, we assume neither $\fd$ or $\fr$ are the XOR or NXOR functions (which, as we know from Proposition~\ref{prop:XOR}, would force the other function to also be of this form). We now show a contradiction arises if either of the functions (without loss of generality, $\fr$) is not forceful. This means we can assume that some index (without loss of generality, $n$), is free for $\fr$.

We claim that for some $i$ and $x \in \{1,-1\}^m$, $\fd$ is \textbf{not} $i$-pivotal at $x$. Otherwise, $\fd$ is $i$-pivotal at every point for every index; it is easy to see by induction that only XOR and NXOR have this property, which we ruled out. Without loss of generality, we can assume $i = m$. Because all the indices are relevant for $\fd$, there must also be at least one $x \in \{1, -1\}^m$ such that $\fd$ is $m$-pivotal at $x$. Thus, we can conclude that $\fd$ is not $m$-pivotal at some $(s_1, \ldots, s_{m-2}, s_{m-1}, *)$ but $\fd$ is $m$-pivotal at some $(s_1', \ldots, s_{m-2}', s_{m-1}', *)$. We now define some more values:
\begin{enumerate}
\item Since every index is relevant for $\fr$, there is some $z_1, \ldots, z_{n-1}$ such that $\fr$ is $n$-pivotal at $(z_1, \ldots, z_{n-1}, *)$. 
\item For each $1 \leq i \leq (m-1)$, because index $n$ is free for $\fr$, there exist $(x_{i,1}, \ldots, x_{i, n-1})$ such that $\fr(x_{i,1}, \ldots, x_{i, n-1}, s_i) = s_i'$.
\item Let $s_m$ be arbitrary and $s_m' = \fr(z_1, z_2, \ldots, z_{n-1}, s_m)$.
\end{enumerate}
We construct the following evaluation of $(\fd, \fr)$:
\[
\begin{matrix}
x_{1,1} & x_{1,2} & \cdots & x_{1,n-1} & s_1 & \rightarrow & s_1' \\
x_{2,1} & x_{2,2} & \cdots & x_{2,n-1} & s_2 & \rightarrow & s_2' \\
\vdots & \vdots & \ddots & \vdots & \vdots & \rightarrow & \vdots \\
x_{m-2,1} & x_{m-2,2} & \cdots & x_{m-2,n-1} & s_{m-2} & \rightarrow & s_{m-2}' \\
x_{m-1,1} & x_{m-1,2} & \cdots & x_{m-1,n-1} & s_{m-1} & \rightarrow & s_{m-1}' \\
z_1 & z_2 & \cdots & z_{n-1} & s_m & \rightarrow & s_m' \\
\downarrow & \downarrow & \downarrow & \downarrow & \downarrow & & \downarrow \\
j_1 & j_2 & \cdots & j_{n-1} & j_n & \rightarrow & j.
\end{matrix}
\]
Here, the rows are correct by our choices of $x_{*,*}$ and $z_*$ and we evaluate the $j_*$ by evaluating $\fd$. 

Now, change $s_m$ to $\negation{s_m}$ in the $(m,n)$-entry of the matrix. Because $\fd$ is not $m$-pivotal at $(s_1, \ldots, s_{m-1}, *)$, $j_n$ does not change. As the first $(n-1)$ columns were unaffected, the other $j_i$ do not change either, so $j$ remains constant. However, because $\fr$ is $n$-pivotal at $(z_1, \ldots, z_{n-1}, *)$, the $(m, n+1)$-entry changes from $s_m'$ to $\negation{s_m'}$. Because $\fd$ is $m$-pivotal at $(s_1',  \ldots, s_{m-1}', *)$, the last column forces $j$ to change to $\negation{j}$. This gives a contradiction.
\end{proof}

\begin{prop}
\label{prop:AND-OR}
Suppose $(\fd, \fr)$ form a nontrivial normal pair where $\fd$ and $\fr$ are forceful functions. Then $\fd$ can only be the AND function (resp. the OR function); $\fr$ must also be the AND function (resp. the OR function).
\end{prop}
\begin{proof}
By Lemma~\ref{lem:forceful}, we can find $c_*$ and $d_*$, all in $\{-1,1\}$, such that
\begin{align*}
\fd(r_1, \ldots, r_m) & = -c_0 + c_0 \frac{1}{2^{m-1}} (1 + c_1r_1) (1 + c_2r_2) \cdots (1 + c_mr_m), \\
\fr(r_1, \ldots, r_n) & = -d_0 + d_0 \frac{1}{2^{n-1}} (1 + d_1r_1) (1 + d_2r_2) \cdots (1 + d_nr_n).
\end{align*}
The Fourier coefficients of $\fd$ are $\h{\fd}(R) = \frac{c_0c_R}{2^{m-1}}$ (where $c_R = \prod_{r \in R}c_r$) for all $R \neq \emptyset$ and $\h{\fd}(\emptyset) = c_0(1/2^{m-1}-1)$; in particular, we know $\h{\fr}(R) \neq 0$ for all $R$. We obtain similar conditions for $\h{\fr}(S)$. Apply Corollary~\ref{cor:powerful} to $R=[m]$ and nonempty $S$ to obtain:
\begin{align*}
(\sum_{S' \supset S} \h{\fr}(S')\h{\fd}(\emptyset)^{|S'-S|})(\frac{c_0c_{[m]}}{2^{m-1}})^{|S|}  & = \h{\fd}([m])\h{\fr}(S)^{m}.
\end{align*}
Pick any row $r \in [m]$ and let $R = [m] \backslash r$. Because $m > 1$, $R$ is nonempty. Corollary~\ref{cor:powerful} gives:
\begin{align*}
(\sum_{S' \supset S} \h{\fr}(S')\h{\fd}(\emptyset)^{|S'-S|})(\frac{c_0c_{[m] \backslash r}}{2^{m-1}})^{|S|}
 & = (\h{\fd}([m]\backslash r) + \h{\fd}([m])\h{\fr}(\emptyset)) \h{\fr}(S)^{m-1}.
\end{align*}
The left-hand sides of the two equations differ (multiplicatively) by $c_r^{|S|}$. Dividing the right-hand sides (using the fact that $\h{\fr}(S)$ is nonzero), we obtain:
\begin{align*}
\h{\fd}([m]\backslash r) + \h{\fd}([m])\h{\fr}(\emptyset) & = \h{\fr}(S) \h{\fd}([m]) c_r^{|S|}, \\
\frac{\h{\fd}([m]\backslash r)}{\h{\fd}([m]) } + \h{\fr}(\emptyset) & = \h{\fr}(S)c_r^{|S|}, \\
c_r + c_0(\frac{1}{2^{n-1}} - 1) & = \frac{d_0d_{S}c_r^{|S|}}{2^{n-1}}, \\ 
(2^{n-1})(c_r - c_0) & = d_0d_{S}c_r^{|S|} - c_0.
\end{align*}
Since $n > 1$ and $c_0$ has the same parity as $c_r$, the left-hand side is divisible by $4$. The right-hand side can only take values in $\{-2, 0, 2\}$, so we must have $c_r = c_0$. Since $r$ was arbitrary, we know that all the $c_i$'s equal; this means $\fd$ must be the AND (if all $c_i$ equal $1$) or the OR function (if all $c_i$ equal $-1$). By symmetry, all the $d_i$'s also equal, meaning $\fr$ must also be the AND or OR function. It is easy to check that $\fr$ is AND if and only if $\fd$ is AND.
\end{proof}

Combining our work in this section, we obtain our main result:
\begin{thm}
\label{thm:main}
Suppose $(\fd, \fr)$ form a nontrivial normal pair. Then only the following cases are possible:
\begin{itemize}
\item $\fd$ and $\fr$ are both AND functions.
\item $\fd$ and $\fr$ are both OR functions.
\item each of $\fd$ and $\fr$ are the XOR or NXOR functions.
\end{itemize}
\end{thm}


\section{Comparison with Existing Literature}
\label{sec:list}

In this section, we use Theorem~\ref{thm:main} about normal pairs to produce some judgment aggregation results. We compare them to existing theorems in the literature of \cite{list2002aggregating}, \cite{list2007}, and \cite{list2013}. 

\subsection{Arrow's Theorem}

First, recall that normal pairs required all the arguments in each functions to be relevant. We can easily reverse this requirement on judges to obtain the following Corollary, which is a closer analogue to Arrow's Theorem:

\begin{cor}
\label{cor:main}
Suppose $\halfagenda$ is symbol-complete and symbol-connected. Then if $f$ is an Arrovian JAR in normal form for $\halfagenda$, at least one of the following must hold:
\begin{itemize}
\item $|\halfagenda| = 1$, in which case $\fr$ can be any UP function.
\item $f$ is a dictator for a single judge $i$, in which case $\halfagenda$ can have any propositions.
\item $|\halfagenda| > 1$ and $f$ is not a dicatorship. Then there is a subset $S$ of judges with $|S| \geq 2$, such that all compound statements in $\halfagenda$ and $\fr$ are simultaneusly the OR, AND, or XOR (in this case, negations are allowed), where $\fr$ is applied to $S$.
\end{itemize}
\end{cor}
\begin{proof}
If there are no compound propositions in $\halfagenda$, then we can only be in the first case. Otherwise, having a compound proposition $\fd$ and restricting to relevant judges gives a normal pair $(\fd, \fr')$ where $\fr'$ is $\fr$ restricted on the relevant judges. We apply Theorem~\ref{thm:main} on $(\fd, \fr')$, remembering the possibility of irrelevant judges when we go back to $\fr$, and the fact that $\fr$ must be UP. Note that if $|\halfagenda| > 1$ and $f$ is not a dictatorship, we cannot have more than one compound proposition type because the OR, AND, and XOR functions are different with $2$ or more arguments.
\end{proof}

We now refer the reader to the two closest related theorems in the literature of Dietrich and List, paraphrased for ease of exposition. 

\begin{thm}[Theorem 2, \cite{list2007}]
\label{thm:list-arrow}
Suppose the agenda $\agenda$ is strongly connected. Then a JAR is Arrovian if and only if it is a dictatorship.
\end{thm} 

The notion of \emph{strongly connected} is fairly strong; we omit it here as we will not need it later. Compared to Theorem~\ref{thm:list-arrow}, our work requires fewer assumptions and thus finds more JAR's. 

To understand the next result, we give a few definitions from \cite{list2013}:
\begin{itemize}
\item An agenda $\agenda$ is \emph{even-number negatable} if there is a minimal inconsistent set $Y \subset \agenda$ with a subset $Z \subset Y$ of even size such that $(Y \backslash Z) \cup \{\negation{p}\colon p \in Z\}$ is consistent.
\item We say that $p \vdash^* q$ if $\{p\} \cup Y$ implies $q$ for some $Y \subset \agenda$ consistent with $p$ and with $\negation{q}$. We say that $p \vdash\vdash^* q$ if there is a sequence of propositions $a_1=p, a_2, \ldots, a_k=q \in \agenda$ such that each $p_i \vdash^* p_{i+1}$. Finally, we say that $\agenda$ is \emph{semi-blocked} if for all propositions $p, q \in \agenda$, [$p \vdash\vdash^* q$ and $q \vdash\vdash^* p$] or [$p \vdash\vdash^* \negation{q}$ and $\negation{q} \vdash\vdash^* p$].
\end{itemize}

\begin{thm}[Theorem 2, \cite{list2013}]
\label{thm:list-arrow-general}
Suppose the agenda $\agenda$ is semi-blocked and even-number negatable. Then if $f$ is an Arrovian JAR, it must be an \emph{oligarchy}, where there is a non-empty set $S$ of the judges and a subset $D \subset \agenda$ such that for all profiles $(p_1, \ldots, p_n)$, the aggregated judgement is AND applied to $S$ for propositions in $\agenda \backslash D$ and OR applied to $S$ for propositions in $D$.
\end{thm}

Theorem~\ref{thm:list-arrow-general} seems more powerful than Theorem~\ref{thm:list-arrow}. The differences between Theorem~\ref{thm:list-arrow-general} and our result Corollary~\ref{cor:main} are:
\begin{itemize}
\item We can check that our symbol-closed condition implies that $\closure(\fd)$ is semi-blocked, so Theorem~\ref{thm:list-arrow-general} applies to a wider framework than Corollary~\ref{cor:main}. 
\item Corollary~\ref{cor:main} catches the XOR case, which is not even-number negatable.
\item Corollary~\ref{cor:main} obtains additional constraints on the logical structure of the propositions; i.e. the propositions themselves must look like OR, AND, and XOR, in addition to the aggregation functions taking that form.  
\item Corollary~\ref{cor:main} notes an enforced ``duality'' between the logical structures of the propositions and the aggregation functions.
\end{itemize}

We were unable to immediately derive our Corollary~\ref{cor:main} from Theorem~\ref{thm:list-arrow-general}. However, Theorem~\ref{thm:list-arrow-general} captures a signficant portion of the power of our result. Thus, an appropriate role of Corollary~\ref{cor:main} is as a nontrivial refinement of a special family (the symbol-closed compound propositions) of Theorem~\ref{thm:list-arrow-general}, giving more explicit information to what the propositions and aggregation functions must be in that family.

\subsection{Anonymity}

We add Anonymity to Corollary~\ref{cor:main} to get a more restrictive theorem and then compare it against literature.

\begin{thm}
\label{thm:our-list-pettit}
Suppose $\halfagenda$ is symbol-complete and symbol-connected. Then if $f$ is a JAR satisfying Anonymity and PI in normal form, $f$ must be the OR, AND, or XOR (or negation) function over all $n$ judges, and each compound proposition $\fd$ must be the same function (or negation) over a subset of the symbols.

\end{thm}
\begin{proof}
First, we may negate $f$ if necessary to ensure $f$ is UP. As $f$ is PI, $f$ is Arrovian. Thus, Corollary~\ref{cor:main} applies. Take a compound statement $\fd$ and its symbols $a_1, \ldots, a_m$. Because of Anonymity, we cannot have any irrelevant judges (since Anonymity means $f$ is symmetric on its arguments). Thus, the only possible JAR's are the OR, AND, and XOR functions (or their negations, as we may have negated $f$ at the beginning of our proof to make $f$ UP) where the $S$ in Corollary~\ref{cor:main} must be the entire set of judges by Anonymity. 
\end{proof}

The most relevant theorem from literature is (paraphrased):
\begin{thm}[Theorem 3, \cite{list2013}]
\label{thm:list-anonymity} Let $n$ be even. If the agenda is blocked, there exists no JAR $f$ such that Anonymity and PI are both satisfied.
\end{thm}
We omit the discussion of \emph{blocked} here. The relevant point is that some of the $\halfagenda$ we allow are blocked and some are not. Thus, Theorem~\ref{thm:our-list-pettit} complements Theorem~\ref{thm:list-anonymity} in the literature, with neither being immediate consequences of the other.

\subsection{Systematicity}

Systematicity is a much stronger condition than PI:
\begin{thm}
\label{thm:our-list-pettit-2}
If $\halfagenda$ contains at least one symbol-closed compound statement $\fd$ and there are $n \geq 2$ judges, then there does not exist a JAR $f$ such that Anonymity and Systematicity are both satisfied.
\end{thm}
\begin{proof}
Note that if we assume Systematicity (instead of PI), we would also need the aggregation functions $f_{\negation{x}} = \flip(f_x)$ to equal $f_x$, by Lemma\ref{lem:PI-doctrinal-strong}. By Theorem~\ref{thm:our-list-pettit}, if suffices to check the OR, AND, and XOR functions (or their negations), and it is easy to check that none of them work for $n \geq 2$.
\end{proof}

This result does strengthen List-Pettit's original result: 
\begin{thm}[Theorem 2, \cite{list2002aggregating}]
\label{thm:list-pettit}
If $\agenda$ contains $\{P, Q, P \wedge Q, \negation{(P \wedge Q)}\}$ and $n \geq 2$, then there does not exist a JAR $f$ such that Anonymity and Systematicity are both satisfied. 
\end{thm}
However, the work from \cite{list2007} and \cite{list2013} basically subsume both of these results. Thus, our Theorem~\ref{thm:our-list-pettit-2} is not a significant contribution to the literature. We present it mostly as a natural followup to the other results in this section.

\section{Conclusion and Future Direction}
\label{sec:conclusion}

We have defined and classified ``normal pairs,'' inspired by Arrow's Theorem. We then obtained some results in the judgment aggregation framework using this concept. We were motivated by the ideal of cooperation and collaboration between different theoretical fields\footnote{Our other recent work in progress in social choice theory with Marengo and Settepanella \cite{marengo-sette-zhang} is similarly motivated, trying to connect the judgment aggregation and classical social choice theory frameworks.}.

It is cute that we see our familiar functions OR, AND, and XOR. Furthermore, we knew from human history that the ``at least one yes means yes'' (OR) and ``at least one no means no'' (AND) rules are reasonable ways\footnote{We are less excited about discovering XOR as a practical tool for judgment aggregation.} for people to aggregate opinions, so it is good that we are not just stuck with dictatorships as in Arrow's Theorem. 

Our work complements the known results in the judgment aggregation literature for the special case of symbol-complete and symbol-connected agendas. Because Arrovian JAR's stay Arrovian on subsets of the propositions, this also provides information for agendas containing a subset logically equivalent to a symbol-complete and symbol-connected agenda. We also show that there is a nice ``duality'' between the roles of the logical propositions and aggregation functions, in that they must use basically the same type of function, but on different numbers of arguments.

Rather than the results themselves, we think the main value of this paper is introducing Fourier-theoretic tools to the standard problems of judgment aggregation. We use these tools to prove Propositions~\ref{prop:AND-OR} and \ref{prop:XOR}. Our proofs are quite short and self-contained, showcasing the power of Fourier analysis. We emphasized ``standard'' because Nehema \cite{nehama2011} has done work on ``approximate judgment aggregation'' where the logical consistency rules are probabilistically correct, also using Fourier analysis. Our works have many similarities (a particular agenda Nehema examines is the \emph{conjunction agenda}, which is the symbol-closed closure of an AND function). Another way of introducing probability into the problem is ``probabilistic judgment aggregation,'' where the judgments themselves contain probabilistic data representing confidence levels, such as in \cite{dietrich2017probabilistic1} and \cite{dietrich2017probabilistic2}. As Fourier analysis can be used on many domains, we suspect Fourier-analytical techniques to be potentially useful there as well.

Compared to the existing literature (which think of logical propositions as abstract set elements, connected by concepts of consistency and inconsistency), our modelling of propositions as elements in a boolean algebra may seem unnecessary. However, the algebraic structure was essential to us realizing the duality between logical propositions and aggregation functions, so we propose it as a useful mental model to have in this field. We are also not losing generality; for example, we can embed the voter preferences in the social choice framework for Arrow's Theorem as logical propositions: $\log_2(m!)$ symbols are enough to embed the strict preference orderings on $m$ candidates (e.g. in some a-priori arbitrary order, ($TTT, TTF, \ldots$) can encode $(a > b > c, a > c > b, \ldots)$ respectively), and then preferences of the form $(x \geq y)$ can be written as binary functions of the preferences.

Our symbol-complete assumption is natural, especially for the ``practical'' law-oriented examples brought up by List-Pettit's original work \cite{list2002aggregating}. However, it is still limiting, and an obvious direction for fruitful future work would be to emulate \cite{list2007} and \cite{list2013}, whose framework extends far beyond symbol-complete assumptions. For example, our above embedding of the preference orderings does not contain a symbol-closed proposition, because the judgments are only done on the pairwise preferences and not the symbols themselves. This means that we do not have an immediate generalization of Arrow's Theorem.

More complex functions would appear if we allow agendas which are not symbol-complete. As an example, let $\halfagenda = \{(X \vee Y), (\negation{X} \vee Y)\}$. Note that the fully rational assignments of values to $\halfagenda$ are exactly those where at least one $T$ is assigned. Then, consider $f$ that assigns to each $x \in \halfagenda$ the ``majority rule'' of a subset of the judges. Because each judge sees at least one $T$, the overall number of $T$'s seen by the relevant judges is at least $50\%$, so at least one aggregated value must be $T$, meaning the result is also consistent. Thus, we must be ready for stranger functions; \cite{list2013} finds some of these functions as well.

We hope that the tools we have created, such as definition of \emph{normal pairs} and the Fourier analysis in Lemma~\ref{lem:powerful-sauce} and Corollary~\ref{cor:powerful}, would be useful for future work. A long-term goal would be an explicit classification of all possible PI aggregation functions for any agenda of logical propositions. 

\section*{Acknowledgments}
We thank Boris Alexeev, Joshua Batson, Paul Christiano, David Dynerman, Steven Karp, and Lauren Williams for helpful conversation. We especially thank Christian List for taking his time to map out the relevant literature for us, and Simona Satepanella for introducing the judgment aggregation framework to us and collaborating on related problems (this work would not have been possible otherwise).

\bibliographystyle{abbrv}
\bibliography{sct}

\end{document}